\documentclass[12pt,oneside,reqno]{amsart}

\usepackage{amsmath, amsthm, amsfonts, amsfonts,amssymb,amscd,amsmath,latexsym, enumerate, verbatim, calc, quotmark, mathtools, txfonts, calrsfs, imakeidx, fancyhdr, listings, lipsum, tkz-euclide, amsthm}
\usepackage{cite}
\usepackage[pagewise]{lineno}
\usepackage{mathtools}

\usepackage[T1]{fontenc}
\usepackage[utf8]{inputenc}

\usepackage[square,sort,comma,numbers]{natbib}
\usepackage[yyyymmdd,hhmmss]{datetime}
\usepackage[all, cmtip]{xy}

\usepackage[pagebackref]{hyperref}
\hypersetup{
colorlinks = true, linkcolor = red,
citecolor   = blue,
urlcolor    = blue,
}

\newtheorem{innercustomgeneric}{\customgenericname}
\providecommand{\customgenericname}{}
\newcommand{\newcustomtheorem}[2]{%
\newenvironment{#1}[1]
{%
	\renewcommand\customgenericname{#2}%
	\renewcommand\theinnercustomgeneric{##1}%
	\innercustomgeneric
}
{\endinnercustomgeneric}
}

\newcustomtheorem{customthm}{Theorem}
\newcustomtheorem{customqn}{Question}

\def\NZQ{\mathbb}               
\def\NN{{\NZQ N}}
\def\QQ{{\NZQ Q}}
\def\ZZ{{\NZQ Z}}

\def\RR{{\NZQ R}}

\def\m{\mathfrak{m}}
\def\n{\mathfrak{n}}

\DeclareMathOperator*{\Ht}{ht}

\DeclareMathOperator*{\Ho}{H}

\DeclareMathOperator*{\depth}{depth}
\DeclareMathOperator*{\Tor}{Tor}
\DeclareMathOperator*{\Ext}{Ext}

\DeclareMathOperator*{\Ch}{char}

\DeclareMathOperator*{\soc}{soc}

\DeclareMathOperator*{\edim}{emb.dim}

\DeclareMathOperator*{\syz}{Syz}

\DeclareMathOperator*{\codim}{codim}
\DeclareMathOperator*{\curv}{curv}
\DeclareMathOperator*{\cx}{cx}
\DeclareMathOperator*{\pd}{pd}

\DeclareMathOperator*{\codepth}{codepth}
\DeclareMathOperator*{\lw}{ll}

\DeclareMathOperator*{\rk}{rank}

\newtheorem{theorem}{Theorem}[section]
\newtheorem{lemma}[theorem]{Lemma}
\newtheorem{corollary}[theorem]{Corollary}
\newtheorem{proposition}[theorem]{Proposition}
\newtheorem{remark}[theorem]{Remark}

\newtheorem{example}[theorem]{Example}

\newtheorem{definition}[theorem]{Definition}

\newtheorem{question}[theorem]{Question}

\newtheorem*{acknowledgement}{Acknowledgement}

\newtheorem*{conj}{Conjecture}
\newtheorem{theoremx}{\bf Theorem}

\title{On Tor-vanishing of local rings}
\date{\today, \currenttime}

\author{Shrikant Shekhar, Anjan Gupta}

\address{Department of Mathematics\\
Indian Institute of Science Education and Research Bhopal\\
Bhopal Bypass Road, Bhopal, Madhya Pradesh, India. Pin - 462066.}
\email{shrikant19@iiserb.ac.in or shrikantshekhar21@gmail.com, anjan@iiserb.ac.in}
\email{}

\thanks{Corresponding author: Shrikant Shekhar; 
{\it email: shrikantshekhar21@gmail.com, shrikant19@iiserb.ac.in}}

\begin{document}

\begin{abstract} 
Let $R$ be a local ring with residue field $k$ and $M$, $N$ be finitely generated modules over $R$. It is well known that $\Tor^R_i(M, N) = 0$ for $i \gg 0$ if 
$\pd_R(M) < \infty$ or $\pd_R(N) < \infty$. The ring $R$ is said to satisfy the Tor-vanishing property if the converse holds, that is, $\Tor^R_i(M, N) = 0$ for $i \gg 0$ implies $\pd_R(M) < \infty$ or $\pd_R(N) < \infty$. Interest in the Tor-vanishing property stems from the fact that Cohen–Macaulay local rings satisfying this property also satisfy the Auslander–Reiten conjecture.

In this article, we study a variant of this property. If $R$ is a generalized Golod ring, we prove that $\Tor^R_i(M, N) = 0$ for $i \gg 0$ implies $\{\curv_R M, \curv_R N \} \cap \{0, 1\} \neq \emptyset$. A key intermediate step in our proof is to show that $\curv_R M \in \{0, 1, \curv_R k\}$ for any module $M$ over a generalized Golod ring $R$. As an application, we prove that generic Gorenstein local rings, {non-trivial} connected sum of generalized Golod-Gorenstein rings satisfy the Tor-vanishing property and consequently the Auslander-Reiten conjecture. Our method suggests a uniform approach and recovers many old results on the Tor-vanishing property.
\end{abstract}

\maketitle



\section{Introduction}
Let $R$ be a commutative Noetherian local ring and $M, N$ be finitely generated modules over  $R$. 
The study of asymptotic vanishing of $\Tor^R_i(M, N), i \geq 1$ and its consequence has long been a central theme of commutative algebra, pioneered by seminal work of Auslander \cite{auslander1999modules} and Lichtenbaum \cite{lichtenbaum1966vanishing}. We are interested to know about a ring $R$ which satisfies the property that $\Tor_i^R(M, N) = 0$ for $i \gg 0$ implies either $M$ or $N$ has finite projective dimension. Rings with this property are said to satisfy the trivial Tor-vanishing property \cite[Definition 3.1]{lyle2020extremal}. Likewise, the ring $R$ is said to satisfy the trivial Ext-vanishing property if $\Ext^i_R(M, N) = 0$ for $i \gg 0$ implies that either the projective dimension of $M$ or the injective dimension of $N$ is finite. For brevity, we refer to the trivial Tor-vanishing property simply as the Tor-vanishing property. If $R$ is a Cohen-Macaulay local ring having a canonical module, then both properties are equivalent \cite[]{lyle2020extremal}. These  vanishing properties have been extensively studied by various authors, viz.  \c{S}ega \cite{csega2003vanishing}, Huneke et al. \cite{huneke2004vanishing}, Nasseh and Wagstaff \cite{nasseh2017vanishing}, Ghosh and Puthenpurakal \cite{ghosh2019vanishing}, Avramov et al.  \cite{avramov2022persistence}, Monta\~{n}o and Lyle \cite{lyle2020extremal}.
Both properties are very strong as \c{S}ega proved that a complete intersection ring of codimension at least two does not satisfy the Tor-vanishing property. 
Our interest in this topic is fueled by the fact that a ring having these properties satisfies a well-known conjecture in commutative algebra and representation theory of algebras, namely the Auslander-Reiten conjecture. The statement below is its version for commutative rings.

\begin{conj}{\rm (Auslander-Reiten conjecture \cite{auslander1975generalized})}{\label{m12}}
Let $R$ be a commutative Noetherian local ring and $M$ be a finitely generated $R$-module. If $\Ext^i_R (M, M \oplus R) = 0$ for every $i > 0$, then $M$ is free.
\end{conj}

There is also another motivation for our interest. A ring $R$ is said to satisfy the Tor-persistence property if $\Tor_i^R(M, M) = 0$ for $i \gg 0$  implies that $M$ has finite projective dimension. An open problem \cite{avramov2022persistence} asks if all rings satisfy the Tor-persistence property. Of course, the Tor-vanishing property implies the Tor-persistence property. The main objective of this article is to find new examples of rings that satisfy the Tor-vanishing property and thereby also satisfy the Tor-persistence property as well as the Auslander-Reiten conjecture.

Now assume that $R$ is a local ring with maximal ideal $\m$, residue field $k$ and $M$ is a module over $R$. The sequence  of Betti numbers $\{\beta_n^R (M)\}$  of $M$ is collected into a more palatable form by considering the Poincar\'e series $P^R_M(t) = \sum_{i = 0}^{\infty} \beta_i^R (M) t^i \in \ZZ[|t|]$. An example attributed to Anick \cite{anick1982counterexample} shows that $P^R_M(t)$ may not be a rational function. Nonetheless, there is a term-wise inequality of formal power series :

$$ \frac{(1+t)^{\edim(R)}}{(1-t^2)^{\rk_k \Ho_1(K^R)}}  \prec P_k^R(t) \prec  \frac{(1+t)^{\edim(R)}}{1 - \sum_{j=1}^{\codepth R} \rk_k \Ho_j(K^R) t^{j+1}},$$
where $K^R $ denotes the Koszul complex on a minimal generating set of the maximal ideal $\m$. 

We consider two extremal cases. When the inequality on LHS is equality, the ring $R$ is a complete intersection. In contrast,  when equality holds on  RHS, the ring $R$ is Golod. Avramov \cite{avramov1994local} found a new class of rings namely generalized Golod rings
of which both complete intersection and Golod rings are examples. He proved that when $R$ is a generalized Golod ring, then there is a polynomial $d_R(t) \in \ZZ[t]$ such that $d_R(t)P^R_M(t) \in \ZZ[t]$ for all $R$-modules $M$. The converse is unknown \cite[Problem 10.3.5]{avramov1998infinite}.

Monta\~{n}o and Lyle studied the Tor-vanishing property of several new examples of generalized Golod rings. The success of their work lies in the following remarkable result. 

\begin{theoremx}\label{MOTLYLE20}{\rm(Monta\~{n}o and Lyle, \cite{lyle2020extremal})}
Let $R$ be an Artinian local ring that is also generalized Golod and set $c =  \codim(R)$, $e= \text{multiplicity of}\ R$,  and $\ell = \min \{ i : \m^{i} = 0 \}$. Assume 
\begin{equation}
e \leq 2c+\ell-3,
\end{equation}
then for any $R$-modules $M$ and $N$, we have
\begin{enumerate}
\item[$(1)$] If $\Tor_i^R(M,N)=0$ for $i \gg 0$, then $\cx_R(M) \leq 1$, or, $\cx_R(N) \leq 1$.
\item[$(2)$] If, additionally, $e \leq 2c+\ell-4$, then $R$ satisfies Tor-vanishing.
\end{enumerate}
\end{theoremx}

The curvature of a module $M$ denoted by $\curv_R M$ is a measure of the exponential growth of the Betti sequence $\{\beta^R_n(M)\}$ (see Definition \ref{curv}). Our purpose in this article is to generalize the above theorem as follows (see Theorem \ref{m8} and Remark \ref{MotPeeva}).

\begin{theoremx}\label{inmain}
Let $R$ be a generalized Golod ring and $M$, $N$ be modules over $R$ such that $\Tor_i^R (M,N) = 0$ for $i \gg 0.$ Then $\{\curv M, \curv N \} \cap \{0, 1\} \neq \emptyset$.
\end{theoremx}

Example \ref{m134} shows that Theorem \ref{inmain} is optimal. As an immediate consequence, we prove that if $R$ is a generalized Golod ring and 
$P^R_k(t)$ does not have a pole at $1$, then $R$ satisfies Tor-vanishing (see Corollary \ref{m10}). 
We find several new rings satisfying the Tor-vanishing property. 
A result of Rossi and \c{S}ega \cite{rossi2014poincare} implies that compressed Artinian Gorenstein local rings of socle degree $s \neq 3$ are generalized Golod rings. Our interest in such rings comes from the fact that generic Gorenstein local rings are compressed \cite{iarrobino1984compressed}. We establish the Tor-vanishing property of such rings in Theorem \ref{AS01}. In the same theorem, we show that certain determinantal ideals of full codimension enjoy the Tor-vanishing property. William Teter \cite{teter1974rings} studied rings of the form $R/ \soc R$  for some Artinian Gorenstein local ring $R$. In literature, such rings are known as Teter rings. We validate the Tor-vanishing of generalized Golod Teter rings in Proposition \ref{cg}. 

Suppose $S$ and $T$ are Artinian local rings with common residue field $k$. The nontrivial fiber product $S \times_k T$ is never Gorenstein even if $S$ and $T$ are so. In \cite{ananthnarayan2012connected}, Ananthnarayan, H. et al. proposed a new construction of Gorenstein rings, namely the connected sums. When $S$ and $T$ are Gorenstein, their connected sum $S \# T$ is defined as the quotient of $S \times_k T$ by the difference $s - t$ of generators $s, t$ of respective socles of $S, T$. Naseeh and Wagstaff \cite{nasseh2017vanishing} proved that a nontrivial fiber product $S \times_k T$  satisfies Tor-vanishing. Analogously, we prove that the nontrivial connected sum $S \# T$ satisfies Tor-vanishing when $S$, $T$ are generalized Golod Artinian Gorenstein local rings (see Theorem \ref{240603}). As an application, we show that if $R$ is an Artinian Gorenstein local ring with maximal ideal $\m$ and $\edim R \geq 5$, then $R$ satisfies Tor-vanishing if either $\m^4 = 0$ or multiplicity of $R$ is at most $11$ (see Theorem \ref{mult11}). Finally, we recover results about the Tor-vanishing property by various authors (see \ref{AS02}).

This paper is organized as follows. In \S-\ref{sec:en}, we recall some basic definitions and results which concern this article.  In \S-\ref{sec:en1}, we prove Theorem \ref{inmain}. A crucial step is to prove that the curvature of a module over a generalized Golod ring lies in $\{0, 1, \curv_R k \}$. The proof uses a classical theorem in real analysis recorded in the book by Titchmarsh \cite[Theorem 7.21]{MR3728294}. This validates a question of Avramov \cite[Problem 4.3.8]{avramov1998infinite} as to if the set $\{ \curv_R M: $ M is a finite $R$-module $\}$ is finite, in the special case when $R$ is a generalized Golod ring; albeit the answer, in general, is negative due to Roos \cite[Corollary 4.5]{roos2003modules}. In the final section \S-\ref{sec:en2}, we provide new examples of rings satisfying the Tor-vanishing property.

{\bf All rings in this article are Noetherian, local with $1\not=0$. Modules are assumed to be finitely generated. The expression ``local ring $(R,\m, k)$" denotes a commutative Noetherian local ring $R$ with maximal ideal $\m$ and residue field $k = R/\m$. When information on the residue field is not necessary, we abbreviate $(R,\m, k)$ as $(R, \m)$.}

\section{Preliminaries}{\label{sec:en}}
In this section, we recall some important definitions and results which we use in the sequel. For unexplained definitions and details we refer the reader to the text by Avramov \cite{avramov1998infinite}.

\begin{definition}{\rm(Betti numbers, Poincar\'e series)}
Let $M$ be an $R$-module. For $n \geq 0$, the $n$-th Betti number $\beta_n^R (M)$ of $M$ is the rank of the free module $F_n$ in the minimal free resolution $F_{*} \twoheadrightarrow M$. It is well known that $\beta_n^R (M) = \dim_k {\Tor}^R_n(k, M)$.

The Poincar\'e series of  $M$ is the generating function for the Betti sequence $\{\beta_n^R (M)\}$, which is defined as $P^R_M(t) = \sum_{n \geq 0} \beta_n^R (M) t^n \in \ZZ[|t|].$
\end{definition}

\begin{definition}{\rm(Golod homomorphisms, Golod rings)}
Let $\phi: R \rightarrow S$ be a surjective homomorphism of local rings with common residue field $k$. From the change of rings spectral sequence for Tor, one gets a term-wise inequality of power series:  $P^S_k(t) \prec \frac{P^R_k(t)}{1 - t(P^R_S(t) -1)}.$
The homomorphism $\phi$ is called a Golod homomorphism if the above inequality is equality, i.e.
\[P^S_k(t) = \frac{P^R_k(t)}{1 - t(P^R_S(t) -1)}.\]

A  local ring $R$ is called a Golod ring if the minimal Cohen presentation $\eta: Q \twoheadrightarrow \hat{R}$ is a Golod homomorphism. Let $K^R$ denote the Koszul complex of $R$ on a minimal set of generators of $\m$. If $R$ is a Golod ring and $e = \dim_k \m/ \m^2$, then one has 
$P^R_k(t) = \frac{(1 + t)^e}{1 - \sum_{i = 1}^e \dim_k H_i(K^R) t^{ i + 1}}$.
\end{definition}

Asymptotic growth of Betti numbers $\{\beta^R_n(M)\}$ is often measured by the reciprocal of the radius of convergence of the Poincar\'e series $P^R_M(t)$ as follows.
\begin{definition}{\rm(Curvature of Modules)}\label{curv}
Let $M$ be an $R$-module. Then the curvature of $M$ is defined as
$ \curv{_R} M = \limsup\limits_{n \rightarrow \infty}  (\beta_n^R (M))^{1/n}.$
\end{definition}

If the projective dimension of $M$ is infinite, then it follows easily that $\curv_R(M) \geq 1$. 
Now we recall the definition of trivial Massey operations \cite[Definition 1.1]{levin1976lectures} which is imperative for the definition of generalized Golod rings.

\begin{definition}{\rm (Trivial Massey operations)}
Let $U$ be a DG algebra. Set $\overline{a}=(-1)^{d + 1}a$ for $a \in U_d$. 
The DG algebra $U$ is said to have a trivial Massey operation if for each sequence $v_1, v_2, \ldots, v_n$ of elements of $\widetilde{H}(U) = \ker (H(U) \rightarrow k)$, there is an element $\gamma(v_1, v_2, \ldots, v_n) \in IU = \ker (U \rightarrow k)$ such that the following hold.

\begin{enumerate}
\item For any $v \in \widetilde{H}(U),$ $\gamma(v)$ is a cycle representing the homology class of $v.$
\item For any $v_1, v_2, \ldots, v_n \in \widetilde{H}(U)$,
$$d \gamma(v_1, v_2, \ldots, v_n) = \sum_{k=1}^{n-1} \overline{\gamma(v_1, \ldots, v_k)} \gamma (v_{k+1}, \ldots, v_n),$$
\end{enumerate}

\end{definition}

\par
\textit{Tate} \cite{tate1957homology} constructed a DG algebra resolution $R \langle X \rangle$ of $k$ over R  by an iterated method of killing cycles by adjoining a set X of exterior variables of odd degrees and divided power variables of even degrees starting from the surjection $R \twoheadrightarrow k$. Subsequently, Gulliksen \cite[Corollary 1.6.4]{gulliksen1969homology} proved that under some extra hypothesis such a resolution becomes minimal and is unique up to isomorphism of DG R-algebras. Following Gulliksen \cite{gulliksen1969homology}, we call such a DG algebra resolution $R \langle X \rangle$ as the acyclic closure $k$ over $R$.

Avramov defined a local ring to be generalized Golod if its homotopy Lie algebra contains a free graded Lie subalgebra of finite codimension \cite[\S 1.7]{avramov1994local}. The following is an equivalent definition.
\begin{definition}{\rm (Generalized Golod rings)}
A local ring $R$ is called generalized Golod (of level $\leq s$) if the DG algebra $R \langle X_{ \leq s} \rangle$ admits a trivial Massey operation, where $R \langle X_{ \leq s} \rangle$ is the DG $R$-subalgebra of the acyclic closure $R\langle X \rangle$ generated by variables of degree at most $s$.
\end{definition}
 The following theorem due to Avramov \cite[\S 1.4]{avramov1994local} proves that finitely generated modules over a generalized Golod ring have rational Poincar\'e series sharing a common denominator.
 
\begin{theorem}{\label{p1}}
Let $M$ be a finitely generated module over a generalized Golod ring $R$  (of level $\leq q$) and $U = R \langle X_i : \deg X_i \leq q \rangle $, then the following hold :
\begin{enumerate}
\item
The series $h(t) = \sum_{i=0}^\infty \dim _k H_i(U)t^i$ is a rational function and 
$h(t)= \frac{u(t)}{v(t)}$ for some $u(t) \in \ZZ[t]$ and $v(t) = \Pi_{0 \leq 2i \leq q} (1-t^{2i})^{e_{2i}}.$

\item
$P_M^R(t) = \frac{p_M(t)}{d_R(t)}$ for  $p_M(t), d_R(t) \in \ZZ[t]$ where
$$d_R(t) = v(t)(1+t-t \cdot h(t)) = (1+t)v(t)-t u(t).$$

\item
$p_k(t) = P^R_k(t) d_R(t) = \Pi_{0 \leq 2i + 1 \leq q} (1 + t^{2i + 1})^{e_{2i + 1}}$, where $e_{2i + 1}$ denotes the number of exterior variables of degree $2i + 1$. 
\end{enumerate}
\end{theorem}
If a ring $R$ has the property that it is a homomorphic image of a complete intersection ring under a Golod homomorphism, then $R$ is a generalized Golod ring of level 2 \cite[\S 1.8(3)]{avramov1994local}. Henceforth, if $R$ is a generalized Golod ring, then $d_R(t)$ denotes the polynomial as in Theorem \ref{p1}.

\section{Main Results}{\label{sec:en1}}
We begin by reminding the reader of the following result which appears in a book by Titchmarsh \cite[Theorem 7.21]{MR3728294}.

\begin{theorem}{\label{m2}}
Let $\frac{f(t)}{g(t)} = \sum_{i \geq 0} a_n t^n$ be a power series expansion of a rational function $\frac{f(t)}{g(t)}$ in $\RR[|t|]$ such that $a_n \geq 0$ for all $n \geq 0$. Then the radius of convergence $R$ is a singular point of $\frac{f(t)}{g(t)}$, i.e. $\lim_{t \rightarrow R}\frac{f(t)}{g(t)} = \infty$.
\end{theorem}

The result below shows that if the Poincar\'e series of a module over a generalized Golod ring has a pole in $(0, 1)$, then the pole is unique and its order is at most one.
\begin{lemma}{\label{m4}}
Let $R$ be a generalized Golod ring. Then $d_R(t)$ has at most one root counting multiplicities in the interval $(0, 1)$.
\end{lemma}

\begin{proof}
We use notations as set in Theorem \ref{p1}.
Let $w(t) = 1+ t - th(t)$, then $d_R(t) = v(t) w(t)$. Since $v(t) = \Pi_{0 \leq 2i \leq q} (1-t^{2i})^{e_{2i}}$ does not have a root in $(0, 1)$, we are through if we show that $w(t)$ has at most one root in $(0, 1)$ counting multiplicities. 

Now $U = R\langle X_{\leq q} \rangle$, by the property of the acyclic closure $R\langle X \rangle $ of $k$ over $R$, we have  $H_0(U) = k$ and $H_i(U) = 0$ for $0 < i < q$. Therefore we have 
\begin{align*}
w(t) = 1+ t - th(t) = 1+ t - t \frac{u(t)}{v(t)} = 1 + t - t \sum_{i=0}^{\infty} \dim _k H_i(U)t^i  = 1 - \sum_{i = q}^{\infty} \dim _k H_i(U)t^{i + 1}
\end{align*}

If $H_i(U) = 0$ for all $i \geq q$, then $w(t) = 1$, so $w(t)$ does not have a root in $(0, 1)$. Therefore we only need to consider the case when $H_i(U) \neq 0$ for some $i = p \geq q$. Since $v(t) = \Pi_{0 \leq 2i \leq q} (1-t^{2i})^{e_{2i}}$, the radius of convergence of $w(t)$ is $1$. Differentiating both sides with respect to $t$, we get 
\[w'(t) = - \sum_{i = q}^{\infty} (i + 1)\dim _k H_i(U)t^{i}.\]
Since $H_p(U) \neq 0$, we have $w'(c) = - \sum_{i = q}^{\infty} (i + 1)\dim _k H_i(U)c^{i}< 0$ for all $c \in (0, 1)$ \cite[Theorem 9.4.10]{bartle2000introduction}. This implies that $w(t)$ is a strictly decreasing continuous function in the interval $(0, 1)$. Therefore $w(t)$ has at most one root in $(0, 1)$ counting multiplicities and the proof follows.
\end{proof}

The following lemma gives an affirmative answer to a problem posed by Avramov \cite[Problem 4.3.8]{avramov1998infinite} for generalized Golod rings.

\begin{lemma}{\label{m5}}
Let $M$ be a module over a generalized Golod ring $R$, then 
${\curv}_R M \in \{0, 1, {\curv}_R k\}.$
\end{lemma}
\begin{proof}
The curvature of a module $M$ is $0$ if and only if the projective dimension of $M$ is finite. On the other hand, 
if $R$ is a complete intersection ring, then every module of infinite projective dimension has curvature equal to $1$ \cite[Corollary 8.2.2]{avramov1998infinite}, so the only case of interest is when $R$ is not a complete intersection ring.

Let ${\curv}_R M \notin \{0, 1\}$, then ${\curv}_R M > 1$. The radius of convergence of $P^R_k(t)$ is $\frac{1}{\curv_R k}$. By Theorem \ref{m2}, $\frac{1}{\curv_R k}$ is a root of $d_R(t)$. In the same vain, $\frac{1}{\curv_R M}$ is also a root of $d_R(t)$.

Since $R$ is not a complete intersection, we have $\curv_R k > 1$ \cite[Corollary 8.2.2]{avramov1998infinite}. Therefore both $\frac{1}{\curv_R k}, \frac{1}{\curv_R M}$ are roots of $d_R(t)$ in the interval $(0, 1)$. The result is followed by Lemma \ref{m4}.
\end{proof}

\begin{theorem}{\label{m8}}
Let $R$ be a generalized Golod ring and  $M$, $N$ be modules over $R$ such that $\Tor_i^R (M,N) = 0$ for $i \gg 0.$ Then either $\curv_R M$ or $\curv_R N$ is in the set $\{ 0, 1 \}.$
\end{theorem}

\begin{proof}
If $R$ is a complete intersection, then the curvature of any module is either $0$ or $1$, so nothing remains to prove. We consider the case when $R$ is not a complete intersection, i.e. $\curv_R k > 1$. Replacing $M$ and $N$ by sufficiently higher syzygies, it is enough to assume that $\Tor_i^R(M, N) = 0 \ \text{for} \ i \geq 1.$

Let $F. \twoheadrightarrow M$ and $G. \twoheadrightarrow N$ be minimal free resolutions of $M$ and $N$ respectively. Then $F. \otimes_R G. \twoheadrightarrow M \otimes_R N$ is the minimal free resolution of $M \otimes_R N$ over $R$, consequently we have $P_{M \otimes_R N}^R (t)= P_M^R(t) P_N^R(t).$

If possible we assume that both $\curv_R M$ and $\curv_R N$ are not in $\{0, 1\}$, then by Lemma \ref{m5} we have 
$${\curv}_R M = {\curv}_R N = {\curv}_R k > 1.$$

Let $\rho = \frac{1}{{\curv}_R k}$.
Then $\rho$ is a common simple pole of both $P_M^R(t)$ and $P_N^R(t)$  in 
$(0, 1)$, see Theorem \ref{m2}. It follows that $\rho \in (0, 1)$ is a pole of $P_{M \otimes_R N}^R (t)$ of order 2, which contradicts Lemma \ref{m4}. Hence, the proof follows.
\end{proof}

\begin{remark}\label{MotPeeva}
Let $(R,\m)$ be an Artinian local ring and $M$ be an $R$-module of infinite projective dimension. Let $c =  \edim(R) - \dim(R) \geq 3$, $e= \text{multiplicity of} \ R$,  and $\ell= \min \{ i : \m^{i} = 0 \}$. If $e \leq 2c+\ell-4$, then it follows from a result of  Peeva \cite[Proposition 2]{peeva1998exponential}, that  $\curv_R(M) > 1$. If $e = 2c+\ell-3$, then the same result implies that if  $\curv_R(M) = 1$ then $\cx_R(M) \leq 1$. 

In light of the above fact, a standard argument on Artinian reduction \cite[Remark 1.7]{csega2003vanishing} and Theorem \ref{m8} readily imply the result of Monta\~{n}o and Lyle \cite[Theorem 4.3]{lyle2020extremal}.
\end{remark}

\begin{corollary}{\label{m10}}
Let $R$ be a generalized Golod ring. If $d_R(1) \neq 0$, then
$R$ satisfies Tor-vanishing. 
\end{corollary}

\begin{proof}
If $d_R(1) \neq 0$, then for any finitely generated module  $M$ over $R$, we have $\curv_R M \neq 1$ by Theorem \ref{m2}. Therefore the corollary follows from Lemmas \ref{m5} and Theorem \ref{m8}.
\end{proof}

The following example delves into the diverse scenarios pertaining to Theorem \ref{m8}.
\begin{example}\label{m134}
We know that $\frac{\QQ[x,y]}{(x^2,xy,y^2)}$ is a Golod ring by \cite[Proposition 5.3.4(1)]{avramov1998infinite}, so the surjective map $p : \QQ[x, y] \twoheadrightarrow   \frac{\QQ[x,y]}{(x^2,xy,y^2)}$ is a Golod homomorphism. Let $T = \frac{\QQ[w, z]}{(w^2,z^2)}$, then $\phi  = T \otimes_{\QQ} p  : \frac{\QQ[w,x,y,z]}{(w^2,z^2)} \rightarrow \frac{\QQ[w,x,y,z]}{(w^2,x^2,xy,y^2,z^2)}$ is a Golod homomorphism. Let $ R = \frac{\QQ[w,x,y,z]}{(w^2,z^2)}$ and $S = \frac{\QQ[w,x,y,z]}{(w^2,x^2,xy,y^2,z^2)}$.

We observe that $S = \frac{\QQ[x,y]}{(x^2,xy,y^2)} \otimes_{\QQ} \frac{\QQ[w]}{(w^2)} \otimes_{\QQ} \frac{\QQ[z]}{(z^2)}$, therefore 
$$P^S_{\QQ}(t) = P^{\frac{\QQ[x,y]}{(x^2,xy,y^2)}}_{\QQ}(t)  P^{\frac{\QQ[w]}{(w^2)}}_{\QQ}(t)  P^{\frac{\QQ[z]}{(z^2)}}_{\QQ}(t)  = \frac{1}{(1 - 2t)(1 - t)^2}$$
The ring $S$ is a generalized Golod ring and $\curv_S\QQ = 2$. We observe that

\begin{enumerate}
\item If $M = \frac{S}{(z)}$, $N = \frac{S}{(w)}$, then $\curv_S M = \curv_S N = 1$ and $\Tor_i^R(M,N)=0$ for all $i \geq 1$.
\item If $M = \frac{S}{(z)}$, $N = \frac{S}{(x,y)}$, then  $\curv_S M = 1$, $\curv_S N = 2$ and $\Tor_i^R(M,N)=0$ for all $i \geq 1$.
\end{enumerate}
\end{example}

\begin{remark}{\label{m133}}
Let  $\phi \colon R \rightarrow S$ be a Golod homomorphism from a complete intersection $R$ onto a ring $S$ such that $\pd_R S > 1$. A module $M$ is called $\phi$-Golod if $P^S_{M} (t) = \frac{P^R_{M}(t)}{1-t(P^R_S(t)-1)}$. For a $\phi$-Golod module $M$, we have the following term-wise inequality of power series : 
\begin{align*}{\label{m131}}
P^S_{M} (t) = \frac{P^R_{M}(t)}{1-t(P^R_S(t)-1)} \succ \frac{1}{1-t(P^R_S(t)-1)} \succ \frac{1}{1-t^2 - t^3}
\end{align*}
The first inequality follows since all coefficients of $P^R_{M}(t)$ are positive and the latter follows because $P^R_S(t) \succ 1 + t + t^2$.
The polynomial $1-t^2 - t^3$ has a root $\gamma \in (0, 1)$, so $\curv_R(M) > \frac{1}{\gamma} > 1$. Note that $S$ is a generalized Golod ring. Therefore by Lemma \ref{m5}, we have 
$\curv_R M = \curv_R k$. More generally, if some higher syzygy $\syz^S_i M, i > 0$ is $\phi$-Golod, then we also have $\curv_R M = \curv_R k$.
\end{remark}

Lescot \cite{lescot1990series}, proved that sufficiently high syzygies of a module over a Golod ring are Golod modules. In Example \ref{m134} (1), we have $\curv_S \QQ \neq \curv_S M$. Therefore by Remark \ref{m133}, higher syzygies of $M$ can't be $\phi$-Golod modules. 
We conclude this section with the following question.

\begin{question}
Let $R$ be a local ring and  $M$, $N$ be modules over $R$ such that $\Tor_i^R (M,N) = 0$ for $i \gg 0.$ Is $\{\curv_R M, \curv_R N\} \cap \{0, 1\} \neq \emptyset$?.
\end{question}

\section{Applications}{\label{sec:en2}}

In this section, we establish the Tor-vanishing and Tor-persistence properties for several new classes of rings. The key is to use Corollary \ref{m10}. In the concluding subsection, we reclaim the Tor-vanishing property of rings, which was previously proved by other authors using different methods. Recall that the socle degree $s$ of an Artinian ring $(R, \m)$ is defined as $s = \max \{n : \m^n \neq 0\}$. In literature, socle degree is also called Loewy length. 

\begin{theorem}{\label{AS01}} 
Let $(R, \m, k)$ be a local ring, then $R$ satisfies Tor-vanishing in the following cases: 
\begin{enumerate}

\item $R$ is a compressed Artinian Gorenstein local ring of socle degree $s$ such that $2 \leq s \neq 3$. 

\item 
$R = A/I$ where $(A, \m, k)$ is a regular local ring, $\Ch(k) = c \geq 0$ and $I=I_1(YX)$ is an ideal of grade $2n, n \geq 2$ for an alternating matrix  $X \in M_{2n + 1}(R)$ and a row vector $Y \in M_{1 \times (2n+1)}(R)$ such that both $X, Y$
have entries in $\m$. 
\end{enumerate}
\end{theorem}

\begin{proof}
The result follows readily by Corollary \ref{m10} once we show that $R$ is a generalized Golod ring and $d_R(1) \neq 0$ in each case.

First, we prove $(1)$. Let $\edim R = e$. If $e \leq 2$, then $R$ is a hypersurface. In this case $R$ satisfies Tor-vanishing by \cite[Theorem 1.9]{huneke1997tensor}. We assume that $e \geq 3$.
By Cohen's structure theorem, there is a regular local ring $(Q, \n)$ such that $R = Q/ I$, $I \subset \n^2$.  A result of Rossi and \c{S}ega \cite[Theorem 5.1]{rossi2014poincare} implies that $R$ is a generalized Golod  ring and $d_R (t) = 1 - t (P^Q_R(t)-1) + t^{e+1}(1 + t)$. It follows that $d_R (1) = 4 - P^Q_R(1)$. The degree of the polynomial $P^Q_R(t) \in \NN[t]$ is equal to  $\pd_Q R = e \geq 3$. The coefficient of $t$ in $P^Q_R(t)$ is $\Tor^Q_1(k, R) = \mu(I)$, where $\mu(I)$ denotes the minimal number of generators of $I$. We observe that $\mu(I) \geq \Ht(I) = \dim Q - \dim R = e \geq 3$. Therefore $P^Q_R(1) \geq 6$. Consequently $d_R(1) \neq 0$ and the statement follows.

To prove $(2)$, we recall that by \cite[Theorem 5.2]{kustin1995huneke}, $R$ is a generalized Golod ring and $d_R(t)$ is given by
\[
d_R(t) =
\begin{cases}
(1+t)^{2n+1}[(1-t)^{2n+1}-t^3], & \ \text{if} \ c=0 \ \text{or} \ n+1 \leq c \\
(1+t)^{2n+1}[(1-t)^{2n+1}(1-t^{2c+1}-t^{2c+2})-t^3], & \ \text{if} \ (n+2)/2 \leq c \leq n.
\end{cases} \]
We observe that $d_R(1) \neq 0$, therefore the result follows.
\end{proof}

The Tor-vanishing property is not assured for compressed Artinian, Gorenstein local rings of socle degree 3 as demonstrated below.

\begin{example}
In each of the following examples, the ring $R$ is an Artinian, Gorenstein compressed ring with $\edim R = 3$, socle degree 3 and length 8.
\begin{enumerate} 
\item $R=\QQ[x, y, z]/(xz,z^2 + xy, y^2 z, x^2, y^3)$. The ring $R$ is not a complete intersection. It is a generalized Golod ring and $P_k^R(t) = \frac{(1 + t)^3}{1 - 5t^2 - 5t^3 + t^5} =\frac{1}{1-3t+t^2}$ \cite[Lemma 4.5]{gupta2020criterion}. It satisfies Tor-vanishing by Corollary \ref{m10}.
\item $R=\QQ[x, y, z]/(x^2, y^2, z^2)$. The ring $R$ is a complete intersection ring of codimension $3$. It does not satisfy Tor-vanishing \cite[Theorem 4.2]{csega2003vanishing}.
\end{enumerate}
\end{example}

The following lemma and the subsequent corollary are imperative to prove the Tor-vanishing property of connected sums of generalized Golod Gorenstein Artinian local rings.

\begin{lemma}{\label{240601}}
Let $R$ be a generalized Golod ring but not a regular ring, then $d_R(1) \leq 0$.
\end{lemma}

\begin{proof}
By Theorem \ref{p1}, we have $d_R(t) = v(t)(1+t-t \cdot h(t)) = (1+t)v(t)-tu(t)$, where $u(t) \in \ZZ[t]$, $v(t) = \Pi_{0 \leq 2i \leq q} (1-t^{2i})^{e_{2i}}$ and $h(t)= \frac{u(t)}{v(t)} \in \ZZ_{\geq 0}[|t|]$. Since $R$ is not a regular ring, we must have $e_2 \geq 1$. This implies that $1$ is the radius of convergence of $h(t)$. For $0 < \epsilon < 1$, we have $v(1 - \epsilon) > 0$. Since $ \frac{u(t)}{v(t)} \in \ZZ_{\geq 0}[|t|] \setminus \{0\}$, we  get $\frac{u(1 - \epsilon)}{v(1 - \epsilon)} > 0$ which implies $u(1 - \epsilon) > 0$. As $0 < \epsilon < 1$ is arbitrary, we obtain $u(1) \geq 0$. Consequently $d_R(1) = - u(1) \leq 0$ and the lemma follows.
\end{proof}

\begin{corollary}{\label{240602}}
Let $(R, \m, k)$ be a generalized Golod ring which is not regular. Then we have $\lim_{t \to 1 } 1/P_k^R(t) \leq 0 $. Moreover, $\lim_{t \to 1 } 1/P_k^R(t) = 0$ if and only if $d_R(1) = 0$.
\end{corollary}

\begin{proposition}\label{cg}
Let $(R, \m, k)$ be an Artinian generalized Golod Gorenstein ring. Then $R/ \soc R$ satisfies Tor-vanishing.
\end{proposition}

\begin{proof} If $\edim R = 1$, then $R = Q/ I$ where $Q$ is a DVR and $I$ is a principal ideal of $Q$. The ring $S = R/ \soc R$ is a hypersurface and so satisfies Tor-vanishing by \cite[Theorem 1.9]{huneke1997tensor}. Now we assume that $\edim R \geq 2$.
It is known by \cite[Theorem 4.6]{gupta2017connection} that $S = R/ \soc R$ is a generalized Golod ring. By \cite[Theorem 2]{levin1978factoring}, we have $\frac{1}{P^S_k(t)} = \frac{1}{P^R_k(t)} - t^2$. Corollary \ref{240602} above yields $\lim_{t \to 1 } \frac{1}{P_k^S(t)} = \lim_{t \to 1 } \frac{1}{P_k^R(t)} - 1 \leq -1$ and $d_S(1) \neq 0$. The result now follows by  Corollary \ref{m10}.
\end{proof}

Let $(S, \m, k)$, $(T, \n, k)$ be Artinian Gorenstein local rings with one-dimensional socles generated by $s, t$ respectively. Let $\pi_S : S \twoheadrightarrow k$, $\pi_T : T \twoheadrightarrow k$ be canonical surjections. 
The connected sum $S \# T$ of the rings $S, T$ is defined as the quotient $\frac{S \times_k T}{(s, - t)}$ of the fiber product $S \times_k T = \{(x, y) : \pi_S(x) = \pi_T(y) \}$ \cite[\S 2]{ananthnarayan2012connected}. The ring $S \# T$ is Gorenstein and $\soc S \# T = ((s, 0))$. Moreover, $\frac{S \# T}{\soc S \# T} \cong \frac{S}{\soc S} \times_k \frac{ T}{\soc T}$.

\begin{theorem}{\label{240603}}
Let $(S, \m, k)$ and $(T, \n, k)$ be generalized Golod Artinian Gorenstein rings of length at least three. Assume $\edim S \geq 2$. Then the connected sum $S \# T$  satisfies Tor-vanishing.
\end{theorem}

\begin{proof}
The connected sum $S \# T$ is a generalized Golod ring \cite[Theorem 5.5]{gupta2017connection}.
Set $\overline{S} = \frac{S}{\soc S}$ and $\overline{T} =  \frac{ T}{\soc T}$. By \cite[Corollary 7.4]{ananthnarayan2012connected}, we have 
\[\frac{1}{P^{S \# T}_k} = \frac{1}{P^{\overline{S}}_k} + \frac{1}{P^{\overline{T}}_k} + t^2 - 1.\]
The argument in  Proposition \ref{cg} gives $\lim_{t \to 0} \frac{1}{P^{\overline{S}}_k} \leq -1$ and Corollary \ref{240602} implies $\lim_{t \to 1} \frac{1}{P^{\overline{T}}_k} \leq 0$. Therefore $\lim_{t \to 0} \frac{1}{P^{S \# T}_k}  \leq -1$. Consequently by Corollary \ref{240602} we have $d_{S \# T}(1) \neq 1$. Hence the proof follows by Corollary \ref{m10}.
\end{proof}

If $R$ is a Gorenstein local ring of multiplicity at most $11$, then $R$ is a generalized Golod ring \cite[Theorem 6.9]{gupta2017connection}. This fact was used by Monta\~{n}o and Lyle to show that such a ring $R$ satisfies Tor-vanishing if  $R$ does not have an embedded deformation \cite[Proposition 4.8]{lyle2020extremal}. The same fact and Theorem \ref{240603} yield part $(2)$ of the theorem below. For an ideal $I$ of a local ring $R$, we use $\mu(I)$ to denote the minimal number of generators of $I$.


\begin{theorem}\label{mult11}
Let $(R, \m)$ be an Artinian Gorenstein local ring satisfying one of the following:

\begin{enumerate}
\item $\m^4 = 0$ and $\mu(\m^2) \leq 4$,
\item $R$ has multiplicity at most 11.
\end{enumerate}
Then $R$ satisfies Tor-persistence. Moreover, if $\edim(R) \geq 5$, then $R$ also satisfies Tor-vanishing.
\end{theorem}

\begin{proof}
If $\edim R \leq 4$, then $R$ satisfies Tor-persistence by \cite[Theorem 4.1(a)(b)]{avramov2022persistence}. We set $\lw R  = \max \{n : \m^n \neq 0\}$. We assume that $\edim R = e \geq 5$ which forces $\lw R \geq 2$ as $R$ is Gorenstein. Suppose $\lw R = 2$, then  $R$ is a stretched Gorenstein ring, i.e. $\mu(\m^2) = 1$. In this case $R$ is a generalized Golod ring and $P^R_k(t) = \frac{1}{1 - et + t^2}$ \cite[Corollary 3.12]{gupta2020criterion}. It follows that $d_R(t) = (1 + t)^e(1 - et + t^2)$ and $d_R(1) = 2^e(2 - e) \neq 0$, so $R$ satisfies Tor-vanishing by Corollary \ref{m10}. Therefore to prove the theorem, we may further assume that $\lw R \geq 3$.

In the setup of both $(1), (2)$, it is proved in \cite[Theorem 6.9]{gupta2017connection} that $R = S \# T$ where $S$, $T$ are Gorenstein local rings such that $\lw S = \lw R \geq 3$ and $\lw T = 2$. Since $\edim R \geq 5$, both $\edim S, \edim T$ cannot be less than $2$. Therefore the result follows by Theorem \ref{240603}.
\end{proof}

\subsection{Recovering old results}{\label{AS02}}
A surjective local ring homomorphism ${\phi} : (Q, \n) \rightarrow (R, \m)$ is said to be an embedded deformation of $R$ if $\ker(\phi)$ is a principal ideal generated by a $Q$-regular element in $\n^2$.
In each of the statements $(1) - (7)$ of the theorem below, the ring $R$ is generalized Golod (see \cite[Application 1.8(1)]{avramov1994local} for (1), \cite[Theorem 6.4]{avramov1988poincare} for (2)-(5), \cite{sjodin1979poincare} for (6), \cite[Corollary 3.12, 3.13]{gupta2020criterion}) for (7)). Moreover, $d_R(1) \neq 0$ in each case (see  \cite[Application 1.8(1)]{avramov1994local} for (1),  \cite[Theorem 3.1]{avramov1989homological} for (2)-(5), \cite[Remark 1.7]{csega2003vanishing} for (6), \cite[Corollary 3.12, 3.13]{gupta2020criterion}) for (7)). Therefore the theorem below follows by Corollary \ref{m10}. Citations are added beside each statement for reference. 

\begin{theorem}
Let $(R, \m)$ be a local ring. Then $R$ satisfies Tor-vanishing in each of the following cases:

\begin{enumerate}
\item $R$ is a Golod ring but not a hypersurface, \cite[Theorem 3.1]{jorgensen1999generalization},
\item $R$ is one link from a complete intersection and $R$ does not have an embedded deformation \cite[Lemma 4.9]{avramov2022persistence},
\item $R$ is two link from a complete intersection, $R$ is Gorenstein and $R$ does not have an embedded deformation \cite[Lemma 4.9)]{avramov2022persistence},
\item $R$ has no embedded deformation and $\edim R - \depth R \leq 3$ \cite[Lemma 4.9]{avramov2022persistence},
\item $R$ is a local Gorenstein ring, $\codim R \leq 4$ and $R$ admits no embedded deformation \cite[Theorem 2.3]{csega2003vanishing} (see \cite[Theorem 3.5, Table 1]{avramov1989homological} for $d_R(t)$),
\item $R$ is a  Gorenstein local ring of minimal multiplicity, $\codim R \geq 3$ \cite[Theorem 3.6]{huneke2003symmetry}.
\item $R$ is an Artinian Gorenstein local ring such that $\edim R = e \geq 3$ and $\mu(\m^2) \leq 2$
\cite[Corollary 3.12, 3.13]{gupta2020criterion}.
\end{enumerate}
\end{theorem}

\begin{remark}
Following  \c{S}ega \cite[1.4]{csega2003vanishing}, we say that a factorization $c(t) = p(t)q(t)r(t)$ in $\ZZ[t]$ is good if $p(t)$ is $1$ or irreducible, $q(t)$ has non-negative coefficients, $r(t)$ is $1$ or irreducible and has no positive real root among its complex roots of minimal absolute value. In \cite[Proposition 1.5]{csega2003vanishing}, \c{S}ega  proved that if $d_R(t)$ has a good factorization, then $R$ satisfies Tor-vanishing.

Statements  $(2) - (7)$ of the above theorem were originally proved by examining the good factorization of $d_R(t)$ which is often a tedious task.  If $c(t) \in \ZZ[t]$ has a good factorization, then it is straightforward to check that $c(1) \neq 0$. Thus Corollary \ref{m10} suggests an alternative approach to study the Tor-vanishing property.
\end{remark}

In view of Proposition \ref{cg} and Theorem \ref{240603}, the following two questions arise naturally.

\begin{question}
Let $(R, \m, k)$ be an Artinian Gorenstein ring. Does $R/ \soc R$ satisfy Tor-vanishing?
\end{question}

\begin{question}
Let $(S, \m, k)$ and $(T, \n, k)$ be Artinian Gorenstein local rings of length at least three. Does the connected sum $S \# T$  satisfy Tor-vanishing?
\end{question}
 
In conclusion, we make the following remark. 
\begin{remark}
In view of the result of Avramov \cite[Theorem 3.1]{avramov1989homological}, an inquisitive mind may enquire if $d_R(1) = 0$ implies that $R$ has an embedded deformation when $R$ is a generalized Golod ring. The answer is negative. In \cite[Proposition 3.1]{gasharov1990boundedness}, Gasharov and Peeva found that if $$R = \frac{\QQ[x_1,x_2,x_3,x_4,x_5]}{(\alpha x_1 x_3 + x_2 x_3,x_1x_4+x_2x_4,x_3^2-x_2x_5 + \alpha x_1x_5,x_4^2-x_2x_5 + x_1 x_5,x_1^2,x_2^2,x_3x_4,x_3x_5,x_4x_5,x_5^2)}, \alpha = 2$$ then $R$ is a Gorenstein local ring with Hilbert series $H(t) = 1 + 5t + 5t^2 + t^3$ and $R$ does not have an embedded deformation \cite[Remark 3.10]{gasharov1990boundedness}. Computations with Macaulay2 \cite{M2} show that $\{x_1, x_1\}$ is an exact pair of zero divisors \cite[Definition Section-1]{HenriquesSega11} and $x_2$ is a Conca generator modulo $x_1R$ \cite[3.6]{HenriquesSega11}, so the ring is generalized Golod and Koszul \cite[Proposition 3.9]{HenriquesSega11} with $P^R_k(t) = \frac{1}{H(-t)}$ \cite[Corollary 4.4 (2)]{HenriquesSega11}. It follows that $d_R(t) = (1 + t)^5H(-t)$, consequently $d_R(1) = 0$.
\end{remark}

\begin{acknowledgement}
We sincerely thank Srikanth B. Iyenger for his comment which helped us find example \ref{m134}. This paper is part of the PhD thesis of the first author. He acknowledges financial support: file number 09/1020(0176)/2019-EMR-I from Council of Scientific and Industrial Research for his PhD. The second author is thankful for \textit{INSPIRE Faculty Fellowship (DST)}, reference no DST/INSPIRE/04/2018/000522.
\end{acknowledgement}

\bibliographystyle{acm}
\bibliography{reference}

\end{document}